\documentclass[10pt,letterpaper]{amsart}
\numberwithin{equation}{section}
\numberwithin{figure}{section}

\usepackage[T1]{fontenc}
\usepackage[utf8]{inputenc}
\usepackage{amssymb}
\usepackage{amsmath}
\usepackage{mathtools}
\usepackage[headings]{fullpage}
\usepackage{thmtools}
\makeatletter
\renewcommand\thmt@autorefsetup{%
  \@xa\def\csname\thmt@envname autorefname\@xa\endcsname\@xa{\thmt@thmname}%
}
\makeatother
\usepackage[dvipsnames]{xcolor}
\usepackage[hidelinks]{hyperref}
\hypersetup{
    colorlinks=true,
    linkcolor=blue,
    urlcolor=red,
    citecolor=Green
}
\DeclareTextFontCommand{\emph}{\color{blue}\em}
\usepackage{enumitem}

\theoremstyle{plain}
\newtheorem{theorem}[equation]{Theorem}
\newtheorem{lemma}[equation]{Lemma}
\newtheorem{corollary}[equation]{Corollary}

\theoremstyle{definition}

\newtheorem{remark}[equation]{Remark}

\newtheorem{example}[equation]{Example}

\newtheorem{definition}[equation]{Definition}

\newtheorem{notation}[equation]{Notation}

\newcommand{\ZZ}{\mathbb{Z}}
\newcommand{\QQ}{\mathbb{Q}}
\newcommand{\FF}{\mathbb{F}}
\newcommand{\LL}{\Lambda_R}

\DeclareMathOperator{\Top}{top}
\newcommand{\Split}{\operatorname{Split}}

\usepackage[backend=biber, style=alphabetic, maxbibnames=999, maxalphanames=999, doi=true, url=false]{biblatex}
\DeclareLabelalphaTemplate{%
  \labelelement{\field[final]{shorthand}}%
  \labelelement{\field[uppercase,strwidth=1,strside=left]{labelname}}%
  \labelelement{\field[strwidth=2,strside=right]{year}\field{extrayear}}%
}
\AtBeginDocument{%
}

\title[Kleber's Conjecture and Complementary Products]{Kleber's conjecture and complementary products of symmetric functions}

\date{July 13, 2026}

\author{Reuven Hodges}
\address{Department of Mathematics, University of Kansas, Lawrence, KS, 66045, USA}
\email{rmhodges@ku.edu}
\author{Hanzhang Yin}
\address{Department of Mathematics, University of Kansas, Lawrence, KS, 66045, USA}
\email{hanyin@ku.edu}

\begin{document}

\begin{abstract}We prove Kleber's rectangular-complement conjecture for Schur functions over
an arbitrary commutative ring $R$, showing that, for a fixed rectangle, the
products $s_\lambda s_{\lambda^\vee}$, indexed by unordered complementary
pairs, are linearly independent in $\Lambda_R$.  The proof rests on a general independence theorem for componentwise
splittings, which asserts that for every partition $\theta$, the products
$s_\alpha s_\beta$ are linearly independent as $\{\alpha,\beta\}$ ranges over
unordered pairs of partitions satisfying $\alpha+\beta=\theta$.  The independence of the products $s_\lambda s_{\lambda^\vee}$ also yields
linear independence of the Koike--Terada universal-character
products over any field, answering a question of Gao--Orelowitz--Yong. We also prove the analogous result for monomial symmetric functions over
fields of characteristic zero, as well as integral linear independence over $\ZZ$.\end{abstract}

\maketitle
\vspace{-8pt}

\section{Introduction}

Schur functions are one of the central bases of the ring of symmetric
functions.  They arise naturally as characters of polynomial representations
of general linear groups, and for partitions $\lambda$ and $\mu$,
multiplication of the corresponding Schur functions defines the
Littlewood--Richardson coefficients $c^\nu_{\lambda,\mu}$ by
\[
s_\lambda s_\mu=\sum_\nu c^\nu_{\lambda,\mu}s_\nu .
\]
Their systematic study goes back to Littlewood and Richardson's work on group
characters and symmetric functions \cite{LittlewoodRichardson1934}.  These
coefficients lie at a meeting point of algebraic combinatorics, representation
theory, and geometry.  They are computed by the Littlewood--Richardson rule and
by crystal, hive, honeycomb, and puzzle models, while the same integers are
tensor-product multiplicities for polynomial representations of $GL_n$ and
structure constants for the cohomology of Grassmannians
\cite{Macdonald1995,Fulton1997,KnutsonTao1999,KnutsonTaoWoodward2004}.
Two basic questions about these expansions are which Schur functions occur and
with what multiplicities.  The Horn inequalities and saturation theorem
characterize nonvanishing \cite{Klyachko1998,KnutsonTao1999}, while
multiplicity-free Schur products have been classified
\cite{Stembridge2001}.  Considering these expansions collectively leads to questions about
linear relations among the products.  This paper studies a linear-independence phenomenon for one distinguished
class of Schur products, namely those obtained by pairing a partition with its
complement in a rectangle.

Let $(b^a)$ denote the $a\times b$ rectangle.  For a partition $\lambda\!=\!(\lambda_1, ..., \lambda_a)\!\subseteq\!(b^a)$, define $\lambda^\vee=(b-\lambda_a,\ldots,b-\lambda_1)$. Equivalently, $\lambda^\vee$ is obtained from the complement of $\lambda$ in the rectangle by rotation through $180^\circ$.  The map $\lambda\mapsto\lambda^\vee$ is an involution, and its orbits are the unordered complementary pairs $\{\lambda,\lambda^\vee\}$.

The same involution is familiar from the Schubert calculus of the Grassmannian $\operatorname{Gr}(a,a+b)$. Under the usual identification of Schubert classes with partitions contained in $(b^a)$, the class $\sigma_{\lambda^\vee}$ is Poincar\'e dual to $\sigma_\lambda$, so the image of the stable product $s_\lambda s_{\lambda^\vee}$ in the Grassmannian quotient is the point class. Kleber's conjecture asserts that this equality of images is not the shadow of a linear dependence in the ring of symmetric functions. Kleber proved the self-complementary case in 2002~\cite{Kleber2002}.  Despite its disarmingly simple statement, the conjecture has remained open for more than two decades; our first theorem establishes the conjectured independence over every commutative ring.

\begin{restatable}[Kleber's conjecture]{theorem}{thmB} \label{thm-rectangular} Let $R$ be a commutative ring. In $\Lambda_R$, for all $a,b\in\ZZ_{>0}$, the products
\[
s_\lambda s_{\lambda^\vee},\qquad \lambda\subseteq (b^a),
\]
indexed by unordered complementary pairs, are linearly independent over $R$. \end{restatable}

The proof of Theorem~\ref{thm-rectangular} proceeds through a complementary problem that is independent of any ambient rectangle.  For a partition $\theta$, a \emph{$\theta$-splitting} is an unordered pair $\{\alpha,\beta\}$ of partitions such that $\alpha+\beta=\theta$ componentwise.  Let $\Split(\theta)$ denote the set of all $\theta$-splittings.

\begin{restatable}{theorem}{thmA} \label{thm-core} Let $R$ be a commutative ring, and let $\theta$ be a partition.  In $\Lambda_R$, the products
\[
s_\alpha s_\beta,
\qquad
\{\alpha,\beta\}\in\Split(\theta),
\]
are linearly independent over $R$. \end{restatable}

Theorem~\ref{thm-core} implies the rectangular theorem by separating rectangular complementary pairs according to the differences between opposite rows.  These differences are unchanged by rectangular complementation.  For a fixed value of this invariant, removing the corresponding forced rows turns a rectangular complementary pair into a splitting of a single partition.  Thus, after choosing a maximal invariant in any hypothetical relation, one obtains a nontrivial relation of the type excluded by Theorem~\ref{thm-core}.

The proof of Theorem~\ref{thm-core} starts from a hypothetical relation among the products indexed by $\theta$-splittings and differentiates with respect to a carefully chosen complete homogeneous generator.  This produces a simpler ordered family, whose independence follows by row removal for Littlewood--Richardson coefficients.  A separate argument handles the only case in which the differentiation does not isolate a coefficient, allowing the argument to remain integral throughout.

The rectangular theorem also gives a Newell--Littlewood consequence.  Let $s_{[\lambda]}$ denote the Koike--Terada universal character associated to $\lambda$.  Over a field $\FF$, these elements form a basis of $\Lambda_\FF$. With respect to the total-degree filtration
\[
F_{\le d}\Lambda_\FF=\bigoplus_{0\le e\le d}(\Lambda_\FF)_e,
\]
the class of $s_{[\lambda]}$ in $F_{\le |\lambda|}\Lambda_\FF/F_{\le |\lambda|-1}\Lambda_\FF$ is the class of $s_\lambda$.  The Newell--Littlewood numbers are the structure constants of the Koike--Terada basis \cite[Section~1.2]{GaoOrelowitzYong2021}.  Gao, Orelowitz, and Yong asked whether the products $s_{[\lambda]}s_{[\lambda^\vee]}$ indexed by complementary pairs in a rectangle are linearly independent, and observed that Kleber's conjecture implies an affirmative answer \cite[Section~7.2]{GaoOrelowitzYong2021}.  We obtain this as a direct corollary of Theorem~\ref{thm-rectangular}.

\begin{restatable}{corollary}{corKT} \label{cor-koike-terada} Let $\FF$ be a field.  For every rectangle $(b^a)$, the products
\[
s_{[\lambda]}s_{[\lambda^\vee]},
\qquad
\lambda\subseteq (b^a),
\]
indexed by unordered complementary pairs, are linearly independent in $\Lambda_\FF$. \end{restatable}

Finally, the same rectangular question can be posed in the monomial basis, with $s_\lambda$ replaced by the monomial symmetric function $m_\lambda$.  We prove the corresponding independence statement by an argument independent of the Schur-function proof; it applies over fields of characteristic zero.  The proof passes to the power-sum basis, where the largest power-sum generator appearing in $m_\lambda$ records $|\lambda|$ and allows one to isolate complementary factors in any putative relation.

\begin{theorem} \label{thm-monomial} Let $\FF$ be a field of characteristic zero.  For every rectangle $(b^a)$, the products
\[
m_\lambda m_{\lambda^\vee},
\qquad
\lambda\subseteq (b^a),
\]
indexed by unordered complementary pairs, are linearly independent in $\Lambda_\FF$. The same products are $\ZZ$-linearly independent in $\Lambda_{\ZZ}$. \end{theorem}

\subsection{Outline}

Section~\ref{sec-prelim} fixes the symmetric-function conventions used throughout, records the coefficient-extraction notation, and proves the first-row and Jacobi--Trudi tools needed later.  Section~\ref{sec-core} proves Theorem~\ref{thm-core} by reducing a hypothetical splitting relation to a derivative-reduced family and then treating the possible maximal self-pair separately. Section~\ref{sec-rectangular} introduces gap vectors and fixed-gap projections, uses them to deduce Theorem~\ref{thm-rectangular} from Theorem~\ref{thm-core}, and proves Corollary~\ref{cor-koike-terada}. Section~\ref{sec-monomial} proves Theorem~\ref{thm-monomial} using M\"obius inversion on the partition lattice and coefficient extraction in the power-sum coordinates.

\section*{Acknowledgements}

We thank Ian Cavey, to whom Kleber's conjecture had been communicated by
Alexander Yong, for bringing the problem to our attention, for helpful
discussions, and for suggesting the analogous question for monomial symmetric
functions.  We also thank Linus Setiabrata for helpful discussions and
Alexander Yong for detailed comments on the manuscript.

\section{Preliminaries} \label{sec-prelim}

This section fixes our notation for symmetric functions and partitions, recalls the tableau form of the Littlewood--Richardson rule used below, and records a Jacobi--Trudi differential calculation needed in the proof of Theorem~\ref{thm-core}.

\subsection{Symmetric functions}

Fix a commutative ring $R$.  We follow the notation and conventions of Macdonald \cite[Chapter~I, \S 1--3]{Macdonald1995}.  Let $\Lambda_{\ZZ}$ be the ring of symmetric functions over $\ZZ$, and set
\[
\Lambda_R=R\otimes_{\ZZ}\Lambda_{\ZZ}.
\]
Unless another coefficient ring is specified, linear independence in $\Lambda_R$ is understood over $R$.  Integer scalars are regarded as elements of $R$ via the canonical map $\ZZ\to R$, $n\mapsto n\cdot 1_R$.  Since all of the results below are immediate when $R$ is the zero ring, we assume throughout that $R$ is nonzero.

A \emph{partition} is a finite weakly decreasing sequence $\lambda=(\lambda_1,\ldots,\lambda_k)$ of positive integers, $\lambda_1\ge\cdots\ge\lambda_k>0$.  We write $\ell(\lambda)=k$ for its length and $|\lambda|=\lambda_1+\cdots+\lambda_k$ for its size; the empty partition $\varnothing$ has length and size zero.  When an expression requires more parts than $\ell(\lambda)$, we pad $\lambda$ with trailing zeroes; conversely, trailing zero parts in a weakly decreasing nonnegative sequence are deleted when the sequence is regarded as a partition.  Thus $\lambda_i=0$ for $i>\ell(\lambda)$, and $\lambda\subseteq\mu$ means $\lambda_i\le\mu_i$ for all $i$.  Let $\bar\lambda=(\lambda_2,\lambda_3,\ldots)$ denote the deletion of the first part of $\lambda$, with $\bar\varnothing=\varnothing$. Lexicographic order on partitions is defined by declaring $\lambda<_{\mathrm{lex}}\mu$ if, at the first index $i$ for which $\lambda_i\ne\mu_i$, one has $\lambda_i<\mu_i$.

For a partition $\lambda$ with $\ell(\lambda)\le n$, write $\lambda^{(n)}=(\lambda_1,\ldots,\lambda_{\ell(\lambda)},0,\ldots,0)$ for the corresponding $n$-tuple.  The \emph{monomial symmetric polynomial} in $n$ variables is
\[
m_\lambda(x_1,\ldots,x_n)
=
\sum_{\alpha} x_1^{\alpha_1}\cdots x_n^{\alpha_n},
\]
where $\alpha$ runs over the distinct permutations of $\lambda^{(n)}$. These polynomials are compatible under the specialization $x_{n+1}=0$; their stable limit is the \emph{monomial symmetric function} $m_\lambda$. The monomial symmetric functions $\{m_\lambda\}$ form an $R$-basis of $\Lambda_R$.

For $r\ge1$, the \emph{complete homogeneous symmetric function} is
\[
h_r=\sum_{1\le i_1\le\cdots\le i_r}x_{i_1}\cdots x_{i_r},
\]
with $h_0=1$ and $h_r=0$ for $r<0$.  If $\lambda=(\lambda_1,\ldots,\lambda_k)$, set $h_\lambda=h_{\lambda_1}\cdots h_{\lambda_k}$, with $h_\varnothing=1$.  The functions $h_1,h_2,\ldots$ are algebraically independent, so $\Lambda_R=R[h_1,h_2,\ldots]$, and the $\{h_\lambda\}$ form an $R$-basis of $\Lambda_R$.

We define Schur functions as follows.  For $n\ge\ell(\lambda)$, the 
\emph{Schur polynomial} in $n$ variables is
\[
s_\lambda(x_1,\ldots,x_n)
=
\frac{\det\bigl(x_i^{\lambda_j+n-j}\bigr)_{1\le i,j\le n}}
     {\det\bigl(x_i^{n-j}\bigr)_{1\le i,j\le n}}.
\]
These polynomials are compatible under the specialization $x_{n+1}=0$, and their stable limit is the \emph{Schur function} $s_\lambda\in\Lambda_{\ZZ}$. The Schur functions $\{s_\lambda\}$ form an $R$-basis of $\Lambda_R$.

We shall use the Jacobi--Trudi identity \cite[Chapter~I, \S3, (3.4)]{Macdonald1995}, which gives
\begin{equation}\label{eq-jacobi-trudi}
s_\lambda
=
\det\bigl(h_{\lambda_i-i+j}\bigr)_{1\le i,j\le r},
\qquad r\ge \ell(\lambda).
\end{equation}

For $r\ge1$, the \emph{power-sum symmetric function} is
\[
p_r=\sum_i x_i^r=m_{(r)}.
\]
For a partition $\lambda=(\lambda_1,\ldots,\lambda_k)$, set $p_\lambda=p_{\lambda_1}\cdots p_{\lambda_k}$, with $p_\varnothing=1$. Over any $\QQ$-algebra $S$, the power sums $\{p_\lambda\}$ form an $S$-basis of $\Lambda_S$.

If $B=\{b_\lambda\}$ is any symmetric function basis and $f=\sum_\lambda a_\lambda b_\lambda$, then $[b_\mu]f=a_\mu$.  We also use coefficient extraction with respect to a chosen polynomial generator.  For example, $[h_a^j]_{\mathrm p}f$ denotes the coefficient of $h_a^j$ after viewing $f\in\Lambda_R=R[h_1,h_2,\ldots]$ as a polynomial in the single variable $h_a$, with coefficients in $R[h_i:i\ge1,\ i\ne a]$.  We write $[h_a]_{\mathrm p}f$ for $[h_a^1]_{\mathrm p}f$.  Similarly, over a field $\FF$ of characteristic zero, $[p_d^j]_{\mathrm p}f$ denotes the coefficient of $p_d^j$ after viewing $f\in\Lambda_\FF=\FF[p_1,p_2,\ldots]$ as a polynomial in $p_d$, with coefficients in $\FF[p_i:i\ge1,\ i\ne d]$.

For nonzero $f\in\LL$, define $\Top(f)$ to be the largest first part among the Schur functions occurring in $f$. In other words,
\[
\Top(f)=\max\{\nu_1\mid [s_\nu]f\ne0\}.
\]

\subsection{The Littlewood--Richardson rule}

For partitions $\lambda$ and $\mu$, define the \emph{Littlewood--Richardson coefficients} $c^\nu_{\lambda,\mu}$ by the Schur-basis expansion
\[
s_\lambda s_\mu=\sum_\nu c^\nu_{\lambda,\mu}s_\nu .
\]
Thus the $c^\nu_{\lambda,\mu}$ are the structure constants for multiplication in the Schur basis.

The \emph{Littlewood--Richardson rule} gives a tableau formula for these structure constants.  Suppose $\lambda\subseteq\nu$.  The \emph{skew shape} $\nu/\lambda$ is the set of boxes $(i,j)$ such that $1\le j\le\nu_i$ and $j>\lambda_i$. A filling of $\nu/\lambda$ is a map $T\colon \nu/\lambda\to\ZZ_{>0}$.  Its content is the sequence $\mu=(\mu_1,\mu_2,\ldots)$, where $\mu_i$ is the number of boxes assigned the value $i$.  The \emph{row word} of $T$ is obtained by reading entries from right to left along each row, starting with the top row and proceeding downward.

The coefficient $c^\nu_{\lambda,\mu}$ is the number of fillings $T$ of $\nu/\lambda$ satisfying
\begin{enumerate}[label=\textup{(\roman*)}]
  \item rows weakly increase from left to right and columns strictly increase from top to bottom;
  \item $T$ has content $\mu$;
  \item the row word of $T$ is a lattice word, meaning that every initial segment contains at least as many $i$'s as $(i+1)$'s, for every $i$.
\end{enumerate}
Such a filling is called a \emph{Littlewood--Richardson tableau}.  In particular, $c^\nu_{\lambda,\mu}=0$ unless $|\nu|=|\lambda|+|\mu|$ and $\lambda\subseteq\nu$.  References include \cite{LittlewoodRichardson1934}, \cite[Chapter~I, \S9]{Macdonald1995}, \cite[Chapter~5]{Fulton1997}, \cite[Chapter~7]{Stanley1999}, and \cite{Stembridge2002}.

\begin{lemma} \label{lem-row-removal}
Let $\alpha,\beta,\nu$ be partitions. \begin{enumerate}[label=\textup{(\roman*)}] \item If $c^\nu_{\alpha,\beta}\ne0$, then $\nu_1\le \alpha_1+\beta_1$. \item If $\nu_1=\alpha_1+\beta_1$, then $c^{(\alpha_1+\beta_1,\bar\nu)}_{\alpha,\beta}=c^{\bar\nu}_{\bar\alpha,\bar\beta}$. \end{enumerate}
\end{lemma}

\begin{proof}
For part~\textup{(i)}, assume $c^\nu_{\alpha,\beta}\ne0$.  By the tableau form of the Littlewood--Richardson rule, there exists a Littlewood--Richardson tableau $T$ of shape $\nu/\alpha$ and content $\beta$.  If the first row of $\nu/\alpha$ is empty, there is nothing to prove.  Otherwise, its rightmost entry is the first letter of the row word, so the lattice condition forces this entry to be $1$.  Since rows weakly increase from left to right, every entry in the first row is equal to $1$. Thus the first row of $\nu/\alpha$ has at most $\beta_1$ boxes, so $\nu_1-\alpha_1\le\beta_1$, and hence $\nu_1\le\alpha_1+\beta_1$.

For part~\textup{(ii)}, this is exactly the $r=1$, $i=j=k=1$ specialization of the row-removal formula of Cho--Moon; see \cite[Theorem~3.1, case $r=1$\textup{(b)}]{CM11}.  We give the direct tableau proof in this special case for completeness.

Assume first that $\beta\ne\varnothing$.  Since $\nu_1=\alpha_1+\beta_1$, the first row of $\nu/\alpha$ has exactly $\beta_1$ boxes.  Let $T$ be a Littlewood--Richardson tableau of shape $\nu/\alpha$ and content $\beta$.  The rightmost entry in the first row is the first letter of the row word, hence is equal to $1$ by the lattice condition.  Since rows weakly increase from left to right, every entry in the first row is equal to $1$.  Deleting this first row and subtracting $1$ from every remaining entry gives a Littlewood--Richardson tableau of shape $\bar\nu/\bar\alpha$ and content $\bar\beta$.

Conversely, start with a Littlewood--Richardson tableau of shape $\bar\nu/\bar\alpha$ and content $\bar\beta$.  Add $1$ to every entry and insert a first row of $\beta_1$ entries equal to $1$ in columns $\alpha_1+1,\ldots,\nu_1$.  The resulting filling has shape $\nu/\alpha$ and content $\beta$.  It is semistandard, and its row word is lattice.  For letters $i\ge2$, this follows from the lattice condition before shifting, while for the pair $1,2$ it follows from $\beta_1\ge\beta_2$.

These two constructions are inverse and hence give a bijection between the Littlewood--Richardson tableaux counted by the two coefficients.  The case $\beta=\varnothing$ is immediate.  Therefore $c^{(\alpha_1+\beta_1,\bar\nu)}_{\alpha,\beta} = c^{\bar\nu}_{\bar\alpha,\bar\beta}$.
\end{proof}

\begin{corollary}\label{cor-top-product}
For partitions $\alpha$ and $\beta$, $\Top(s_\alpha s_\beta)=\alpha_1+\beta_1$.
\end{corollary}

\begin{proof}
By Lemma~\ref{lem-row-removal}\textup{(i)}, every Schur function $s_\nu$ occurring in $s_\alpha s_\beta$ has $\nu_1\le\alpha_1+\beta_1$, so $\Top(s_\alpha s_\beta)\le\alpha_1+\beta_1$.

For the reverse inequality, note that $c^{\bar\alpha+\bar\beta}_{\bar\alpha,\bar\beta}=1$.  Indeed, the tableau of shape $(\bar\alpha+\bar\beta)/\bar\alpha$ obtained by filling every box added in row $i$ with the entry $i$ is a Littlewood--Richardson tableau of content $\bar\beta$.  It is the only one, since the lattice condition forces the first row to contain only entries equal to $1$, and since that row has $\bar\beta_1$ boxes, it contains all entries equal to $1$.  Proceeding row by row, once the entries $1,\ldots,i-1$ have been accounted for, the lattice condition forces the rightmost entry in row $i$ to be $i$, and row semistandardness then forces the whole row to be filled with $i$.  Since $(\bar\alpha+\bar\beta)_1\le\alpha_1+\beta_1$, the sequence $(\alpha_1+\beta_1,\bar\alpha+\bar\beta)$ is a partition.  Hence Lemma~\ref{lem-row-removal}\textup{(ii)} gives
\[
c^{(\alpha_1+\beta_1,\bar\alpha+\bar\beta)}_{\alpha,\beta}
=
c^{\bar\alpha+\bar\beta}_{\bar\alpha,\bar\beta}
=
1.
\]
Thus a Schur function with first part $\alpha_1+\beta_1$ occurs in $s_\alpha s_\beta$, so $\Top(s_\alpha s_\beta)\ge\alpha_1+\beta_1$. \end{proof}

\subsection{A Jacobi--Trudi derivative}

We use the following formal derivative on the polynomial presentation $\Lambda_R=R[h_1,h_2,\ldots]$.  For $n>0$, let $\partial/\partial h_n$ be the unique $R$-derivation satisfying
\[
\frac{\partial h_i}{\partial h_n}
=
\begin{cases}
1, & i=n,\\
0, & i\ne n.
\end{cases}
\]
Since $\partial/\partial h_n$ is an $R$-derivation, we shall use the usual linearity and Leibniz rules without further comment.

\begin{definition}
Let $\tau$ be a nonempty partition.  Let
\[
M_\tau=\bigl(h_{\tau_i-i+j}\bigr)_{1\le i,j\le \ell(\tau)}.
\]
Thus \eqref{eq-jacobi-trudi}, with $r=\ell(\tau)$, gives $s_\tau=\det M_\tau$.  For $1\le p,q\le \ell(\tau)$, let $M_\tau^{(p,q)}$ denote the matrix obtained from $M_\tau$ by deleting row $p$ and column $q$; when $\ell(\tau)=1$, deleting the only row and column leaves the $0\times0$ matrix, whose determinant is $1$.

Let $\partial\tau$ be the partition obtained from $(\tau_2-1,\tau_3-1,\ldots,\tau_{\ell(\tau)}-1)$ by deleting zero parts. \end{definition}

\begin{lemma} \label{lem-jt-derivative}
Let $\tau$ be a nonempty partition.
\begin{enumerate}[label=\textup{(\roman*)}]
\item Viewed as a polynomial in $h_{\tau_1+\ell(\tau)-1}$, the Schur function $s_\tau$ has degree one. \item One has $[h_{\tau_1+\ell(\tau)-1}]_{\mathrm{p}} s_\tau = (-1)^{\ell(\tau)-1}s_{\partial\tau}$. \item For every integer $k\ge \tau_1+\ell(\tau)-1$,
\[
\frac{\partial s_\tau}{\partial h_k}
=
\begin{cases}
(-1)^{\ell(\tau)-1}s_{\partial\tau},
& k=\tau_1+\ell(\tau)-1,\\
0,
& k>\tau_1+\ell(\tau)-1.
\end{cases}
\]
\end{enumerate}
\end{lemma}

\begin{proof}
For the $(i,j)$-entry of $M_\tau$, one has $\tau_i-i+j\le \tau_1-i+\ell(\tau)\le \tau_1+\ell(\tau)-1$.  Equality holds only when $i=1$ and $j=\ell(\tau)$.  Hence no entry of $M_\tau$ involves $h_k$ for $k>\tau_1+\ell(\tau)-1$, so $\partial s_\tau/\partial h_k=0$ for such $k$.

It remains to isolate the coefficient of $h_{\tau_1+\ell(\tau)-1}$.  Taking the cofactor expansion of $\det M_\tau$ along the last column gives
\begin{equation}\label{eq-jt-last-column-expansion}
s_\tau
=
(-1)^{1+\ell(\tau)}
h_{\tau_1+\ell(\tau)-1}\det M_\tau^{(1,\ell(\tau))}
+
\sum_{i=2}^{\ell(\tau)}
(-1)^{i+\ell(\tau)}
h_{\tau_i-i+\ell(\tau)}
\det M_\tau^{(i,\ell(\tau))}.
\end{equation}
The only entry of $M_\tau$ involving $h_{\tau_1+\ell(\tau)-1}$ is the $(1,\ell(\tau))$-entry.  Since every cofactor $M_\tau^{(i,\ell(\tau))}$ appearing in \eqref{eq-jt-last-column-expansion} is obtained by deleting the last column, none of these determinants involves $h_{\tau_1+\ell(\tau)-1}$.  For $i\ge2$, we also have $\tau_i-i+\ell(\tau)\le\tau_1+\ell(\tau)-2$, so the displayed factor $h_{\tau_i-i+\ell(\tau)}$ in the summation term is not $h_{\tau_1+\ell(\tau)-1}$. Hence, in \eqref{eq-jt-last-column-expansion}, the only occurrence of $h_{\tau_1+\ell(\tau)-1}$ is the displayed linear factor in the first summand.  Therefore $s_\tau$ has degree at most one in $h_{\tau_1+\ell(\tau)-1}$, and
\begin{equation}\label{eq-jt-h-coefficient-cofactor}
[h_{\tau_1+\ell(\tau)-1}]_{\mathrm{p}} s_\tau
=
(-1)^{1+\ell(\tau)}\det M_\tau^{(1,\ell(\tau))}.
\end{equation}

If $\ell(\tau)>1$, then after deleting row $1$ and column $\ell(\tau)$ the remaining matrix is
\[
\bigl(h_{(\tau_{i+1}-1)-i+j}\bigr)_{1\le i,j\le \ell(\tau)-1}.
\]
This is the Jacobi--Trudi matrix for $\partial\tau$, padded with trailing zeros to length $\ell(\tau)-1$, so its determinant is $s_{\partial\tau}$. If $\ell(\tau)=1$, then $\partial\tau=\varnothing$ and $M_\tau^{(1,1)}$ is the $0\times0$ matrix, whose determinant is $1=s_\varnothing$.  Thus, in all cases,
\begin{equation}\label{eq-jt-leading-cofactor}
\det M_\tau^{(1,\ell(\tau))}=s_{\partial\tau}.
\end{equation}
Substituting \eqref{eq-jt-leading-cofactor} into \eqref{eq-jt-h-coefficient-cofactor} proves \textup{(ii)}.  Since $s_{\partial\tau}$ is a Schur-basis element, this coefficient is nonzero; hence $s_\tau$ has degree one in $h_{\tau_1+\ell(\tau)-1}$, proving \textup{(i)}.

For $k=\tau_1+\ell(\tau)-1$, parts \textup{(i)} and \textup{(ii)} give
\[
\frac{\partial s_\tau}{\partial h_{\tau_1+\ell(\tau)-1}}
=
[h_{\tau_1+\ell(\tau)-1}]_{\mathrm{p}} s_\tau
=
(-1)^{\ell(\tau)-1}s_{\partial\tau}.
\]
The case $k>\tau_1+\ell(\tau)-1$ was proved in the first paragraph, so \textup{(iii)} follows.
\end{proof}

\section{The \texorpdfstring{$\theta$}{theta}-splitting theorem} \label{sec-core}

This section proves Theorem~\ref{thm-core}.  The key intermediate result is an independence statement for the derivative terms that arise from a hypothetical relation among the products $s_\alpha s_\beta$ indexed by $\theta$-splittings.

\begin{definition}\label{def-first-row-projection}
For an integer $n$, define $\pi_n$ on the Schur basis by
\[
\pi_n(s_\nu)=
\begin{cases}
s_{\bar\nu}, & \nu_1=n,\\
0, & \nu_1\ne n,
\end{cases}
\]
and extend $R$-linearly to $\Lambda_R$. \end{definition}

\begin{lemma} \label{lem-top-row-projection}
For possibly empty partitions $\alpha$ and $\beta$,
\[
\pi_{\Top(s_\alpha s_\beta)}(s_\alpha s_\beta)
=
s_{\bar\alpha}s_{\bar\beta}.
\]
\end{lemma}

\begin{proof} If one of $\alpha$ and $\beta$ is empty, the claim follows directly from the definition of $\pi_n$.  Thus assume that both are nonempty. Expanding gives
\[
\pi_{\Top(s_\alpha s_\beta)}(s_\alpha s_\beta)
=
\pi_{\Top(s_\alpha s_\beta)}
\!\left(\sum_\nu c^\nu_{\alpha,\beta}s_\nu\right)
=
\sum_\gamma c^{(\alpha_1+\beta_1,\gamma)}_{\alpha,\beta}s_\gamma
=
\sum_\gamma c^\gamma_{\bar\alpha,\bar\beta}s_\gamma
=
s_{\bar\alpha}s_{\bar\beta},
\]
where the second equality follows from $R$-linearity and the definition of $\pi_{\Top(s_\alpha s_\beta)}$, using Corollary~\ref{cor-top-product} to identify $\Top(s_\alpha s_\beta)=\alpha_1+\beta_1$, and the third equality is Lemma~\ref{lem-row-removal}\textup{(ii)}.
\end{proof}

\subsection{A derivative-reduced splitting family}

The Jacobi--Trudi derivative breaks the symmetry between the two factors of a splitting; accordingly, the independence statement below is formulated for ordered decompositions of $\theta$.

\begin{lemma}\label{lem-oriented-block}
Fix a partition $\theta$ and an integer $n$.  The functions
\[
s_\alpha s_{\partial\beta},
\]
indexed by ordered pairs $(\alpha,\beta)$ of partitions such that $\alpha+\beta=\theta$ componentwise, $\beta\ne\varnothing$, and $\beta_1+\ell(\beta)-1=n$, are linearly independent over $R$. \end{lemma}

\begin{proof}
Let $r=\ell(\theta)$.  We prove the lemma by induction on $r$.

If $r=1$, then any indexed pair has $\beta=(n)$, so there is at most one term.  If such a term occurs, it is $s_\alpha s_\varnothing=s_\alpha$. Since $s_\alpha$ is an element of the Schur $R$-basis of $\Lambda_R$, any relation $c\,s_\alpha=0$ has $c=0$.  Thus the family is linearly independent.

Assume $r>1$, and suppose that the statement is known for partitions of length $r-1$.  Consider a relation
\begin{equation}\label{eq-oriented-relation}
\sum_{\substack{(\alpha,\beta):\, \alpha+\beta=\theta\\
\beta\ne\varnothing,\ \beta_1+\ell(\beta)-1=n}}
c_{\alpha,\beta}\,s_\alpha s_{\partial\beta}=0.
\end{equation}

Assume, for contradiction, that some coefficient with $\ell(\beta)\ge2$ is nonzero.  Among such nonzero coefficients, choose a minimal value
\[
t=\min\{\alpha_2+r-\ell(\beta)\mid c_{\alpha,\beta}\ne0,\ \ell(\beta)\ge2\}.
\]
For any indexed pair with $\ell(\beta)\ge2$, Corollary~\ref{cor-top-product} and the conditions $\alpha+\beta=\theta$ and $\beta_1+\ell(\beta)-1=n$ give
\begin{equation}\label{eq-oriented-top}
\Top(s_\alpha s_{\partial\beta})
=
\alpha_1+\beta_2-1
=
\theta_1+\theta_2-n+r-2-(\alpha_2+r-\ell(\beta)).
\end{equation}
Set $N=\theta_1+\theta_2-n+r-2-t$.  By \eqref{eq-oriented-top}, every term with $\alpha_2+r-\ell(\beta)>t$ has top first row strictly smaller than $N$, and every term with $\alpha_2+r-\ell(\beta)<t$ has zero coefficient by the choice of $t$.

By $R$-linearity, applying $\pi_N$ to \eqref{eq-oriented-relation} gives another relation in $\Lambda_R$.  We first determine which summands have nonzero image. Terms with $\ell(\beta)\ge2$ and $\alpha_2+r-\ell(\beta)>t$ have top first row strictly smaller than $N$ by \eqref{eq-oriented-top}, and hence have zero image under $\pi_N$. Terms with $\ell(\beta)\ge2$ and $\alpha_2+r-\ell(\beta)<t$ have zero coefficient by the choice of $t$.

The possible term with $\ell(\beta)=1$ also has zero image under $\pi_N$. Indeed, such a term has $\beta=(n)$, hence $\alpha_1=\theta_1-n$, and its summand is $c_{\alpha,\beta}s_\alpha$.  On the other hand, for any indexed pair with $\ell(\beta)\ge2$ and $\alpha_2+r-\ell(\beta)=t$, we have $N-(\theta_1-n)=\beta_2+\ell(\beta)-2>0$. Thus the possible length-one summand has Schur first part $\theta_1-n<N$, so its image under $\pi_N$ is zero.

Therefore only the terms with $\ell(\beta)\ge2$ and $\alpha_2+r-\ell(\beta)=t$ can have nonzero image under $\pi_N$.  Fix such a pair.  Since \eqref{eq-oriented-top} gives $N=\Top(s_\alpha s_{\partial\beta})$, Lemma~\ref{lem-top-row-projection} gives
\[
\pi_N(s_\alpha s_{\partial\beta})
=
s_{\bar\alpha}s_{\overline{\partial\beta}}
=
s_{\bar\alpha}s_{\partial\bar\beta},
\]
where the final equality follows directly from the definition of $\partial$. Thus applying $\pi_N$ to \eqref{eq-oriented-relation} gives
\begin{equation}\label{eq-tail-relation}
\sum_{\substack{(\alpha,\beta):\, \alpha+\beta=\theta\\
\beta_1+\ell(\beta)-1=n\\
\ell(\beta)\ge2,\ \alpha_2+r-\ell(\beta)=t}}
c_{\alpha,\beta}\,s_{\bar\alpha}s_{\partial\bar\beta}=0.
\end{equation}

Each pair in \eqref{eq-tail-relation} satisfies $\bar\alpha+\bar\beta=\bar\theta$ and $\bar\beta\ne\varnothing$.  The condition $\alpha_2+r-\ell(\beta)=t$ is equivalent to
\[
\bar\beta_1+\ell(\bar\beta)-1
=
\beta_2+\ell(\beta)-2
=
\theta_2+r-2-t.
\]
Thus all terms in \eqref{eq-tail-relation} belong to the family appearing in Lemma~\ref{lem-oriented-block} for the partition $\bar\theta$ and the integer $\theta_2+r-2-t$.

By the induction hypothesis applied to $\bar\theta$, the coefficient of each indexed product in \eqref{eq-tail-relation} is zero, after combining terms with the same ordered pair $(\bar\alpha,\bar\beta)$.  We now check that this combining does not merge coefficients from distinct pairs in the original relation.  The map $(\alpha,\beta)\mapsto(\bar\alpha,\bar\beta)$ is injective on the pairs appearing in \eqref{eq-tail-relation}.  Indeed, from $\bar\beta$ and the fixed integer $n$ one recovers $\beta_1=n-\ell(\bar\beta)$, hence $\beta$, and then $\alpha$ is determined by $\alpha+\beta=\theta$.  Therefore each coefficient seen by the induction hypothesis is a single original coefficient $c_{\alpha,\beta}$, and all coefficients appearing in \eqref{eq-tail-relation} are zero.

Hence all coefficients with $\ell(\beta)\ge2$ and $\alpha_2+r-\ell(\beta)=t$ are zero, contradicting the choice of $t$.  Thus every coefficient with $\ell(\beta)\ge2$ is zero.

It remains to consider terms with $\ell(\beta)=1$.  The conditions force $\beta=(n)$, and then $\alpha$ is determined by $\alpha+\beta=\theta$. Thus there is at most one such ordered pair.  If it occurs, then $\partial\beta=\varnothing$, so the remaining summand is $c_{\alpha,\beta}s_\alpha$.  Since $s_\alpha$ is a Schur basis element, $c_{\alpha,\beta}s_\alpha=0$ implies $c_{\alpha,\beta}=0$.
\end{proof}

\begin{lemma}\label{lem-maximal-self-pair}
Let $\tau$ be a nonempty partition, and set $n=\tau_1+\ell(\tau)-1$.  Suppose
\begin{equation}\label{eq-maximal-self-pair-relation}
c\,s_\tau^2+
\sum_{i=1}^u d_i\,s_{\alpha^{(i)}}s_{\beta^{(i)}}=0
\end{equation}
in $\Lambda_R$, where $u\ge0$, $c,d_i\in R$, and, for each $i$, either $\alpha^{(i)}=\varnothing$ or $\alpha^{(i)}_1+\ell(\alpha^{(i)})-1<n$, and either $\beta^{(i)}=\varnothing$ or $\beta^{(i)}_1+\ell(\beta^{(i)})-1<n$.  Then $c=0$.
\end{lemma}

\begin{proof}
By Lemma~\ref{lem-jt-derivative}\textup{(i)} and \textup{(ii)},
\[
[h_n^2]_{\mathrm{p}}s_\tau^2
=
\bigl([h_n]_{\mathrm{p}}s_\tau\bigr)^2
=
s_{\partial\tau}^2.
\]

Let $\lambda$ be one of the partitions $\alpha^{(i)}$ or $\beta^{(i)}$. If $\lambda=\varnothing$, then $s_\lambda=1$.  If $\lambda\ne\varnothing$, then $\lambda_1+\ell(\lambda)-1<n$ by hypothesis.  Taking $r=\ell(\lambda)$ in \eqref{eq-jacobi-trudi}, every entry in the determinant for $s_\lambda$ is of the form $h_{\lambda_a-a+b}$, where $1\le a,b\le \ell(\lambda)$.  Its index satisfies $\lambda_a-a+b\le \lambda_1+\ell(\lambda)-1<n$, while entries with negative index are $0$ and entries with index $0$ are $1$.  Hence $s_\lambda\in R[h_j\mid 1\le j<n]$.  Therefore $[h_n^2]_{\mathrm p}s_{\alpha^{(i)}}s_{\beta^{(i)}}=0$ for every $i$.

Thus, applying $[h_n^2]_{\mathrm p}$ to \eqref{eq-maximal-self-pair-relation} gives $c\,s_{\partial\tau}^2=0$. By the Littlewood--Richardson rule, $[s_{\partial\tau+\partial\tau}]s_{\partial\tau}^2$ counts Littlewood--Richardson tableaux of shape $(\partial\tau+\partial\tau)/(\partial\tau)$ and content $\partial\tau$. There is exactly one such tableau, since the lattice condition forces the boxes added in row $a$ to be filled with $a$, and this row-constant filling is semistandard and has lattice reading word because $\partial\tau$ is a partition.  Thus $c^{\partial\tau+\partial\tau}_{\partial\tau,\partial\tau}=1$.  Hence
\[
0=[s_{\partial\tau+\partial\tau}]\bigl(c\,s_{\partial\tau}^2\bigr)
=
c\,c^{\partial\tau+\partial\tau}_{\partial\tau,\partial\tau}
=
c.\qedhere
\]
\end{proof}

\subsection{Proof of the core theorem}

The derivative-reduced splitting family and the maximal self-pair lemma now give the proof of Theorem~\ref{thm-core}.

\begin{proof}[Proof of Theorem~\ref{thm-core}.]
Suppose, for contradiction, that there is a nontrivial relation
\begin{equation}
\label{eq-core-relation}
\sum_{\{\alpha,\beta\}\in\Split(\theta)}
c_{\alpha,\beta}\,s_\alpha s_\beta=0.
\end{equation}
If $\theta=\varnothing$, then the only $\theta$-splitting is $\{\varnothing,\varnothing\}$, and \eqref{eq-core-relation} becomes $c_{\varnothing,\varnothing}s_\varnothing^2=0$.  Since $s_\varnothing=1$, this implies $c_{\varnothing,\varnothing}=0$.  Thus \eqref{eq-core-relation} is trivial, contrary to assumption. Hence $\theta\ne\varnothing$, so every $\theta$-splitting has at least one nonempty part. Set
\[
n=\max\{\delta_1+\ell(\delta)-1\mid
\{\alpha,\beta\}\in\Split(\theta),\ c_{\alpha,\beta}\ne0,\
\delta\in\{\alpha,\beta\},\ \delta\ne\varnothing\}.
\]
The maximum is taken over a nonempty finite set, since \eqref{eq-core-relation} is nontrivial and every $\theta$-splitting has at least one nonempty part. Let
\[
\operatorname{High}_n=
\{\delta\ne\varnothing\mid
\delta_1+\ell(\delta)-1=n,\
\delta\in\{\alpha,\beta\}\text{ for some }
\{\alpha,\beta\}\in\Split(\theta)\text{ with }c_{\alpha,\beta}\ne0\}.
\]
Then $\operatorname{High}_n\ne\varnothing$ by the definition of $n$.

Applying $\partial/\partial h_n$ to \eqref{eq-core-relation} gives
\begin{equation}\label{eq-core-derivative-product-rule}
0=
\sum_{\{\alpha,\beta\}\in\Split(\theta)}
c_{\alpha,\beta}\left(
\frac{\partial s_\alpha}{\partial h_n}s_\beta
+
s_\alpha\frac{\partial s_\beta}{\partial h_n}
\right).
\end{equation}
Consider one of the two summands coming from a fixed $\{\alpha,\beta\}\in\Split(\theta)$ with $c_{\alpha,\beta}\ne0$.  Let $s_\delta$ be the differentiated factor and let $s_\gamma$ be the other factor; explicitly, $(\gamma,\delta)=(\beta,\alpha)$ for the first summand inside the parentheses in \eqref{eq-core-derivative-product-rule}, and $(\gamma,\delta)=(\alpha,\beta)$ for the second.  Then $\gamma+\delta=\theta$.  If $\delta=\varnothing$, then the differentiated factor is $s_\varnothing=1$, so the summand is zero.  If $\delta\ne\varnothing$, then the maximality of $n$ and Lemma~\ref{lem-jt-derivative}\textup{(iii)} show that the summand is zero unless $\delta\in\operatorname{High}_n$.  In the remaining case, Lemma~\ref{lem-jt-derivative}\textup{(iii)} gives the contribution
\[
(-1)^{\ell(\delta)-1}c_{\gamma,\delta}\,
s_\gamma s_{\partial\delta},
\]
where $c_{\gamma,\delta}$ is the coefficient $c_{\alpha,\beta}$ in \eqref{eq-core-derivative-product-rule} for the unordered splitting $\{\alpha,\beta\}=\{\gamma,\delta\}$.

Substituting these surviving contributions into \eqref{eq-core-derivative-product-rule} gives
\begin{equation}
\label{eq-core-derivative-oriented}
0=
\sum_{\substack{(\gamma,\delta):\,\gamma+\delta=\theta\\
\gamma\ne\delta,\ \delta\in\operatorname{High}_n}}
(-1)^{\ell(\delta)-1}c_{\gamma,\delta}\,
s_\gamma s_{\partial\delta}
+
\sum_{\substack{\tau\in\operatorname{High}_n\\ \tau+\tau=\theta}}
2(-1)^{\ell(\tau)-1}c_{\tau,\tau}\,
s_\tau s_{\partial\tau}.
\end{equation}
The factor $2$ in the second sum comes from the case $\alpha=\beta$, where the two summands inside the parentheses in \eqref{eq-core-derivative-product-rule} are equal.

Every indexed product in \eqref{eq-core-derivative-oriented} satisfies the hypotheses of Lemma~\ref{lem-oriented-block}.  For the first sum, use the ordered pair $(\gamma,\delta)$; since $\delta\in\operatorname{High}_n$, we have $\delta\ne\varnothing$ and $\delta_1+\ell(\delta)-1=n$.  For the second sum, use the ordered pair $(\tau,\tau)$, and the same conditions follow from $\tau\in\operatorname{High}_n$.  Thus Lemma~\ref{lem-oriented-block} applies to \eqref{eq-core-derivative-oriented}, so every coefficient in \eqref{eq-core-derivative-oriented} is zero.  In particular,
\begin{equation}\label{eq-core-unequal-max-vanishing}
c_{\gamma,\delta}=0
\quad
\text{whenever } \gamma\ne\delta,\ \gamma+\delta=\theta,\
\delta\in\operatorname{High}_n.
\end{equation}
Equation \eqref{eq-core-unequal-max-vanishing} eliminates every $\delta\in\operatorname{High}_n$ whose other part in the $\theta$-splitting is different from $\delta$.  If no $\tau\in\operatorname{High}_n$ satisfied $\tau+\tau=\theta$, this would imply $\operatorname{High}_n = \varnothing$, which is a contradiction. Hence there exists $\tau\in\operatorname{High}_n$ with $\tau+\tau=\theta$. It is unique, since $\tau+\tau=\theta$ determines $\tau$.  For this remaining case, Lemma~\ref{lem-oriented-block} applied to \eqref{eq-core-derivative-oriented} gives only $2(-1)^{\ell(\tau)-1}c_{\tau,\tau}=0$, which does not imply $c_{\tau,\tau}=0$ over an arbitrary ring $R$.

By \eqref{eq-core-unequal-max-vanishing}, all terms in \eqref{eq-core-relation} whose indexing splitting contains an element of $\operatorname{High}_n$ are zero except possibly $c_{\tau,\tau}s_\tau^2$. Therefore \eqref{eq-core-relation} has the form
\[
c_{\tau,\tau}s_\tau^2
+
\sum_{\substack{\{\alpha,\beta\}\in\Split(\theta)\\
\text{each nonempty }\delta\in\{\alpha,\beta\}\text{ satisfies }
\delta_1+\ell(\delta)-1<n}}
c_{\alpha,\beta}\,s_\alpha s_\beta
=0.
\]
Since $n=\tau_1+\ell(\tau)-1$, Lemma~\ref{lem-maximal-self-pair} gives
\begin{equation}\label{eq-core-self-max-vanishing}
c_{\tau,\tau}=0.
\end{equation}
Equations \eqref{eq-core-unequal-max-vanishing} and \eqref{eq-core-self-max-vanishing} imply $\operatorname{High}_n=\varnothing$, since any $\delta\in\operatorname{High}_n$ belongs to a splitting $\{\gamma,\delta\}$ with nonzero coefficient, while \eqref{eq-core-unequal-max-vanishing} eliminates the case $\gamma\ne\delta$ and uniqueness of $\tau$ together with \eqref{eq-core-self-max-vanishing} eliminates the case $\gamma=\delta$. This contradicts $\operatorname{High}_n\ne\varnothing$.  Hence no nontrivial relation \eqref{eq-core-relation} exists.
\end{proof}

\section{Rectangular complements} \label{sec-rectangular}

We now use Theorem~\ref{thm-core} to prove Theorem~\ref{thm-rectangular}.  The remaining point is to show that, after successively removing the saturated first rows determined by a fixed gap vector, the rectangular complement products become one of the componentwise splitting families covered by Theorem~\ref{thm-core}.

If $a=1$, then the rectangular complementary pairs are precisely the $(b)$-splittings, so Theorem~\ref{thm-rectangular} follows directly from Theorem~\ref{thm-core}.  We may therefore assume that $a\ge2$.

Fix integers $a\ge2$ and $b\ge1$, and put $m=\lfloor a/2\rfloor$.  For $\lambda\subseteq (b^a)$, define
\[
\operatorname{gap}(\lambda)
=
(\lambda_1-\lambda_a,\lambda_2-\lambda_{a-1},\ldots,
\lambda_m-\lambda_{a+1-m}).
\]
The coordinates of $\operatorname{gap}(\lambda)$ are weakly decreasing and nonnegative, so $\operatorname{gap}(\lambda)$ is a partition.  Moreover,
\[
\operatorname{gap}(\lambda^\vee)
=
\bigl((b-\lambda_a)-(b-\lambda_1),\ldots,
(b-\lambda_{a+1-m})-(b-\lambda_m)\bigr)
=
\operatorname{gap}(\lambda).
\]

For a partition $\rho$ and an integer $r\ge0$, set $\rho^{\downarrow r}=(\rho_{r+1},\rho_{r+2},\ldots,\rho_{\ell(\rho)})$.  For $\lambda\subseteq(b^a)$, we call $\lambda^{\downarrow m}$ its tail.

Now fix a gap vector $g=(g_1,\ldots,g_m)$.  If $\operatorname{gap}(\lambda)=g$, then
\begin{equation}\label{eq-fixed-gap-tail-sum}
\lambda^{\downarrow m}+(\lambda^\vee)^{\downarrow m}
=
\begin{cases}
(b-g_m,b-g_{m-1},\ldots,b-g_1), & a=2m,\\
(b,b-g_m,b-g_{m-1},\ldots,b-g_1), & a=2m+1
\end{cases}
=:\Theta_g.
\end{equation}
In particular, the sum $\lambda^{\downarrow m}+(\lambda^\vee)^{\downarrow m}$ depends only on $\operatorname{gap}(\lambda)$.

For $g=(g_1,\ldots,g_m)$, let
\[
\Pi_g=\pi_{b+g_m}\circ\cdots\circ\pi_{b+g_2}\circ\pi_{b+g_1},
\]
where $\pi_n$ is the first-row projection from Definition~\ref{def-first-row-projection}.

\begin{example} Let $a=5$, $b=7$, and $\lambda=(6,5,4,3,2)$.  Then $m=2$ and
\[
\operatorname{gap}(\lambda)=(6-2,5-3)=(4,2).
\]
The rectangular complement is $\lambda^\vee=(5,4,3,2,1)$, and $\operatorname{gap}(\lambda^\vee)=(4,2)$ as well.  For $g=(4,2)$, we have
\[
\lambda^{\downarrow m}=(4,3,2),
\qquad
(\lambda^\vee)^{\downarrow m}=(3,2,1),
\qquad
\lambda^{\downarrow m}+(\lambda^\vee)^{\downarrow m}=(7,5,3).
\]

In this case $\Pi_g=\pi_9\circ\pi_{11}$.  Since $6+5=11$, Lemma~\ref{lem-top-row-projection} gives
\[
\pi_{11}(s_\lambda s_{\lambda^\vee})
=
s_{(5,4,3,2)}s_{(4,3,2,1)}.
\]
Since $5+4=9$, applying Lemma~\ref{lem-top-row-projection} again gives
\[
\Pi_g(s_\lambda s_{\lambda^\vee})
=
\pi_9\bigl(s_{(5,4,3,2)}s_{(4,3,2,1)}\bigr)
=
s_{(4,3,2)}s_{(3,2,1)}
=
s_{\lambda^{\downarrow m}}s_{(\lambda^\vee)^{\downarrow m}}.
\]
Thus, in this example, applying $\Pi_g$ gives $s_{\lambda^{\downarrow m}}s_{(\lambda^\vee)^{\downarrow m}}$.  The next lemma proves that this holds for every $\lambda$ with fixed gap vector $g$, and that the resulting tail pairs have a componentwise sum that depends only on $g$. \end{example}

\begin{lemma}\label{lem-gap-block} Fix a gap vector $g=(g_1,\ldots,g_m)$.

\begin{enumerate} \item[(i)] If $\lambda\subseteq(b^a)$ satisfies $\operatorname{gap}(\lambda)=g$, then $\Pi_g(s_\lambda s_{\lambda^\vee}) =s_{\lambda^{\downarrow m}}s_{(\lambda^\vee)^{\downarrow m}}$, and $\{\lambda^{\downarrow m},(\lambda^\vee)^{\downarrow m}\}\in\Split(\Theta_g)$.

\item[(ii)] If $\lambda\subseteq(b^a)$ satisfies $\operatorname{gap}(\lambda)<_{\mathrm{lex}} g$, then $\Pi_g(s_\lambda s_{\lambda^\vee})=0$. \end{enumerate}
\end{lemma}

\begin{proof}
By \eqref{eq-fixed-gap-tail-sum}, $\{\lambda^{\downarrow m},(\lambda^\vee)^{\downarrow m}\}\in \Split(\Theta_g)$.

We prove by induction on $r$, for $0\le r\le m$, that applying $\pi_{b+g_r}\circ\cdots\circ\pi_{b+g_1}$ to $s_\lambda s_{\lambda^\vee}$ gives $s_{\lambda^{\downarrow r}}s_{(\lambda^\vee)^{\downarrow r}}$, with the case $r=0$ interpreted as the identity.  Suppose $1\le r\le m$ and the claim holds for $r-1$.  The first rows of $\lambda^{\downarrow(r-1)}$ and $(\lambda^\vee)^{\downarrow(r-1)}$ are $\lambda_r$ and $b-\lambda_{a+1-r}$, whose sum is $b+g_r$.  Lemma~\ref{lem-top-row-projection} therefore gives
\[
\pi_{b+g_r}\bigl(
s_{\lambda^{\downarrow(r-1)}}s_{(\lambda^\vee)^{\downarrow(r-1)}}
\bigr)
=
s_{\lambda^{\downarrow r}}s_{(\lambda^\vee)^{\downarrow r}}.
\]
This proves the induction.  The case $r=m$ is precisely $\Pi_g(s_\lambda s_{\lambda^\vee}) =s_{\lambda^{\downarrow m}}s_{(\lambda^\vee)^{\downarrow m}}$.

Now suppose $h=\operatorname{gap}(\lambda)<_{\mathrm{lex}}g$, and let $j$ be the first index such that $h_j\ne g_j$.  Then $h_i=g_i$ for $i<j$ and $h_j<g_j$.  Applying the first $j-1$ projections gives $s_{\lambda^{\downarrow(j-1)}}s_{(\lambda^\vee)^{\downarrow(j-1)}}$ by the induction just proved, with $h$ in place of $g$.  The first rows of $\lambda^{\downarrow(j-1)}$ and $(\lambda^\vee)^{\downarrow(j-1)}$ are $\lambda_j$ and $b-\lambda_{a+1-j}$, whose sum is $b+h_j$.  Hence the first-row bound in Lemma~\ref{lem-row-removal} implies that every Schur term in $s_{\lambda^{\downarrow(j-1)}}s_{(\lambda^\vee)^{\downarrow(j-1)}}$ has first part at most $b+h_j$.  Since $b+h_j<b+g_j$, applying $\pi_{b+g_j}$ gives zero.  Therefore $\Pi_g(s_\lambda s_{\lambda^\vee})=0$. \end{proof}

\begin{lemma}\label{lem-fixed-gap-injective}
Fix a gap vector $g=(g_1,\ldots,g_m)$.  For each $\lambda\subseteq(b^a)$ with $\operatorname{gap}(\lambda)=g$, the map
\[
\{\lambda,\lambda^\vee\}
\longmapsto
\{\lambda^{\downarrow m},(\lambda^\vee)^{\downarrow m}\}
\]
is injective.
\end{lemma}

\begin{proof}
Suppose $\lambda,\mu\subseteq(b^a)$ have gap vector $g$ and
\[
\{\lambda^{\downarrow m},(\lambda^\vee)^{\downarrow m}\}
=
\{\mu^{\downarrow m},(\mu^\vee)^{\downarrow m}\}.
\]
Thus $\lambda^{\downarrow m}=\mu^{\downarrow m}$ or $\lambda^{\downarrow m}=(\mu^\vee)^{\downarrow m}$.

\noindent\textit{Case 1.} $\lambda^{\downarrow m}=\mu^{\downarrow m}$. Then $\lambda_i=\mu_i$ for all $i>m$.  Since $m=\lfloor a/2\rfloor$, the index $a+1-i$ is greater than $m$ whenever $1\le i\le m$.  Hence $\lambda_{a+1-i}=\mu_{a+1-i}$ for $1\le i\le m$.  Since $\operatorname{gap}(\lambda)=\operatorname{gap}(\mu)=g$, we have $\lambda_i=g_i+\lambda_{a+1-i}$ and $\mu_i=g_i+\mu_{a+1-i}$ for $1\le i\le m$.  Thus $\lambda_i=\mu_i$ for $1\le i\le m$, and therefore $\lambda=\mu$.

\noindent\textit{Case 2.} $\lambda^{\downarrow m}=(\mu^\vee)^{\downarrow m}$. Then $\lambda_i=(\mu^\vee)_i$ for all $i>m$.  Since $m=\lfloor a/2\rfloor$, the index $a+1-i$ is greater than $m$ whenever $1\le i\le m$.  Hence $\lambda_{a+1-i}=(\mu^\vee)_{a+1-i}$ for $1\le i\le m$.  Also $\operatorname{gap}(\mu^\vee)=\operatorname{gap}(\mu)=g$, so $\operatorname{gap}(\lambda)=\operatorname{gap}(\mu^\vee)=g$.  Therefore $\lambda_i=g_i+\lambda_{a+1-i}$ and $(\mu^\vee)_i=g_i+(\mu^\vee)_{a+1-i}$ for $1\le i\le m$.  It follows that $\lambda_i=(\mu^\vee)_i$ for $1\le i\le m$, and hence $\lambda=\mu^\vee$.

In either case, $\{\lambda,\lambda^\vee\}=\{\mu,\mu^\vee\}$, so the map is injective.
\end{proof}

We now prove the rectangular theorem by choosing a lexicographically maximal gap vector in a hypothetical relation and projecting onto its fixed-gap block.

\begin{proof}[Proof of Theorem~\ref{thm-rectangular}.]
Suppose, for contradiction, that there is a nontrivial relation
\begin{equation}
\label{eq-rectangular-relation}
\sum_{\{\lambda,\lambda^\vee\}} c_\lambda\,s_\lambda s_{\lambda^\vee}=0,
\end{equation}
where the sum is over unordered complementary pairs.  Since $\operatorname{gap}(\lambda)=\operatorname{gap}(\lambda^\vee)$, the set
\[
\mathcal G
=
\{\operatorname{gap}(\lambda)\mid c_\lambda\ne0\}
\]
is well-defined and nonempty.  Choose a lexicographically maximal element $g=(g_1,\ldots,g_m)$ of $\mathcal G$, and let
\[
\mathcal B_g
=
\{\{\lambda,\lambda^\vee\}\mid c_\lambda\ne0,\
\operatorname{gap}(\lambda)=g\}.
\]

If $\{\lambda,\lambda^\vee\}\notin\mathcal B_g$ has nonzero coefficient, then $\operatorname{gap}(\lambda)<_{\mathrm{lex}}g$ by the choice of $g$.  Hence Lemma~\ref{lem-gap-block}\textup{(ii)} gives $\Pi_g(s_\lambda s_{\lambda^\vee})=0$.  Applying $\Pi_g$ to \eqref{eq-rectangular-relation} therefore gives
\begin{equation}\label{eq-rectangular-projected-relation}
0
=
\sum_{\{\lambda,\lambda^\vee\}\in\mathcal B_g}
c_\lambda\,\Pi_g(s_\lambda s_{\lambda^\vee})
=
\sum_{\{\lambda,\lambda^\vee\}\in\mathcal B_g}
c_\lambda\,
s_{\lambda^{\downarrow m}}s_{(\lambda^\vee)^{\downarrow m}},
\end{equation}
where the second equality is Lemma~\ref{lem-gap-block}\textup{(i)}.

By Lemma~\ref{lem-gap-block}\textup{(i)}, each unordered pair $\{\lambda^{\downarrow m},(\lambda^\vee)^{\downarrow m}\}$ lies in $\Split(\Theta_g)$.  By Lemma~\ref{lem-fixed-gap-injective}, these unordered pairs are distinct as $\{\lambda,\lambda^\vee\}$ ranges over $\mathcal B_g$.  Since every coefficient indexed by $\mathcal B_g$ is nonzero, \eqref{eq-rectangular-projected-relation} is a nontrivial relation among products indexed by a subset of $\Split(\Theta_g)$, contradicting Theorem~\ref{thm-core}.
\end{proof}

\begin{proof}[Proof of Corollary~\ref{cor-koike-terada}.]
Suppose that there is a relation
\[
\sum_{\{\lambda,\lambda^\vee\}} c_\lambda\,
s_{[\lambda]}s_{[\lambda^\vee]}=0,
\]
where the sum is over unordered complementary pairs in $(b^a)$.  By the filtration property recalled above, the top homogeneous component of $s_{[\lambda]}$ is $s_\lambda$.  Since $|\lambda|+|\lambda^\vee|=ab$ for every $\lambda\subseteq(b^a)$, the homogeneous component of degree $ab$ in $s_{[\lambda]}s_{[\lambda^\vee]}$ is $s_\lambda s_{\lambda^\vee}$. Taking the degree-$ab$ homogeneous component of the displayed relation gives
\[
\sum_{\{\lambda,\lambda^\vee\}} c_\lambda\,
s_\lambda s_{\lambda^\vee}=0.
\]
By Theorem~\ref{thm-rectangular}, all coefficients $c_\lambda$ are zero. Thus the Koike--Terada products are linearly independent. \end{proof}

\section{Monomial complements} \label{sec-monomial}

This section proves Theorem~\ref{thm-monomial}; throughout, $\FF$ is a field of characteristic zero and $\FF^\times$ is the multiplicative group of nonzero elements of $\FF$.

Let $\mathsf{Par}_k$ be the partition lattice of $\{1,\ldots,k\}$, ordered by refinement.  Thus an element $\pi\in\mathsf{Par}_k$ is a collection of nonempty disjoint subsets, called \emph{blocks}, whose union is $\{1,\ldots,k\}$, and $\pi\le_{\mathrm{ref}}\sigma$ means that every block of $\pi$ is contained in a block of $\sigma$.

\begin{example}
In $\mathsf{Par}_4$, the set partition $\{\{1,3\},\{2,4\}\}$ has blocks $\{1,3\}$ and $\{2,4\}$.  Also $\{\{1\},\{2\},\{3,4\}\}\le_{\mathrm{ref}}\{\{1,2\},\{3,4\}\}$, since each block on the left is contained in a block on the right.
\end{example}

The minimum of $\mathsf{Par}_k$ is the all-singletons partition $\widehat0=\{\{1\},\ldots,\{k\}\}$, and its maximum is the one-block partition $\widehat1=\{\{1,\ldots,k\}\}$.  For $\sigma\in\mathsf{Par}_k$, write $p\sim_\sigma q$ if $p$ and $q$ lie in the same block.  Let $\mu$ be the M\"obius function of the poset $\mathsf{Par}_k$.

\begin{lemma}
\label{lem-monomial-top-p} Let $\alpha= (\alpha_1,\ldots,\alpha_k)$ be a nonempty partition such that $|\alpha|=d$.  Then
\[
m_\alpha=\kappa_\alpha p_d+Q_\alpha,
\qquad
Q_\alpha\in \FF[p_1,\ldots,p_{d-1}],
\]
with $\kappa_\alpha\in\FF^\times$.
\end{lemma}

\begin{proof}
Let $r_j(\alpha)$ be the multiplicity of the part $j$ in $\alpha$, and set $r_\alpha=\prod_{j\ge1}r_j(\alpha)!$.

For $\sigma\in\mathsf{Par}_k$, let $X_\sigma(\alpha)$ be the sum over $k$-tuples $(i_1,\ldots,i_k)$ of positive integers whose equality pattern is exactly $\sigma$, that is
\[
X_\sigma(\alpha)
=
\sum_{\substack{i_1,\ldots,i_k\ge1\\
u\sim_\sigma v \Longleftrightarrow i_u=i_v}}
x_{i_1}^{\alpha_1}\cdots x_{i_k}^{\alpha_k}.
\]
Thus, $X_{\widehat0}(\alpha)$ is the sum over $k$-tuples with pairwise distinct entries.  Every monomial occurring in this sum has exponent partition $\alpha$, so it occurs in $m_\alpha$.  Conversely, every monomial occurring in $m_\alpha$ is obtained from such a tuple.  Two tuples occurring in $X_{\widehat0}(\alpha)$ give the same monomial exactly when one is obtained from the other by a permutation $\omega$ of $\{1,\ldots,k\}$ such that $\alpha_{\omega(u)}=\alpha_u$ for every $u$. There are $\prod_{j\ge1}r_j(\alpha)!=r_\alpha$ such permutations.  Hence each monomial of $m_\alpha$ appears $r_\alpha$ times in $X_{\widehat0}(\alpha)$, and therefore
\begin{equation}\label{eq:X0-m}
X_{\widehat0}(\alpha)=r_\alpha m_\alpha.
\end{equation}

For $\pi\in\mathsf{Par}_k$, let $Y_\pi(\alpha)$ be the sum where the equalities prescribed by $\pi$ are required, but no distinctness condition is imposed between different blocks
\[
Y_\pi(\alpha)
=
\sum_{\substack{i_1,\ldots,i_k\ge1\\
u\sim_\pi v\Rightarrow i_u=i_v}}
x_{i_1}^{\alpha_1}\cdots x_{i_k}^{\alpha_k}.
\]
If $B$ is a block of $\pi$, set $\alpha_B=\sum_{u\in B}\alpha_u$.  A $k$-tuple $(i_1,\ldots,i_k)$ satisfies the condition in the definition of $Y_\pi(\alpha)$ exactly when there is a choice of positive integers $j_B$, one for each block $B$ of $\pi$, such that $i_u=j_B$ for every $u\in B$. Different blocks may have the same value of $j_B$.  Thus
\[
Y_\pi(\alpha)
=
\sum_{\substack{j_B\ge1\\ B\in\pi}}
\prod_{B\in\pi}x_{j_B}^{\alpha_B}
=
\prod_{B\in\pi}\sum_{j_B\ge1}x_{j_B}^{\alpha_B}
=
\prod_{B\in\pi}p_{\alpha_B}.
\]

Now partition the defining sum for $Y_\pi(\alpha)$ disjointly according to exact equality pattern.  If a $k$-tuple $(i_1,\ldots,i_k)$ has exact equality pattern $\sigma$, then it satisfies the condition defining $Y_\pi(\alpha)$ exactly when $u\sim_\pi v$ implies $u\sim_\sigma v$.  Equivalently, every block of $\pi$ is contained in a block of $\sigma$, so $\pi\le_{\mathrm{ref}}\sigma$.  Therefore
\begin{equation}\label{eq:Ypi-Xsigma}
Y_\pi(\alpha)=\sum_{\sigma\ge_{\mathrm{ref}}\pi}X_\sigma(\alpha).
\end{equation}
Applying M\"obius inversion to \eqref{eq:Ypi-Xsigma} on $\mathsf{Par}_k$ \cite[Section~3.7, Proposition~3.7.2]{Stanley2012} gives
\[
X_{\widehat0}(\alpha)
=
\sum_{\pi\in\mathsf{Par}_k}\mu(\widehat0,\pi)Y_\pi(\alpha)
=
\sum_{\pi\in\mathsf{Par}_k}\mu(\widehat0,\pi)
\prod_{B\in\pi}p_{\alpha_B}.
\]
The one-block partition contributes $\mu(\widehat0,\widehat1)p_d=(-1)^{k-1}(k-1)!p_d$ \cite[Section~3.10, Example~3.10.4, equation~(3.37)]{Stanley2012}.  Every other partition $\pi$ has at least two blocks, so each block sum $\alpha_B$ is strictly between $0$ and $d$.  Therefore every other term lies in $\FF[p_1,\ldots,p_{d-1}]$.  Since $\FF$ has characteristic zero, $r_\alpha$ is invertible in $\FF$. Dividing \eqref{eq:X0-m} by $r_\alpha$ gives the result, with $\kappa_\alpha=(-1)^{k-1}(k-1)!/r_\alpha\in\FF^\times$.
\end{proof}

\begin{proof}[Proof of Theorem~\ref{thm-monomial}.]
We prove the linear independence statement over $\FF$.  Let
\[
\mathcal P^{\mathrm{rep}}_{a,b}
=
\left\{
\lambda\subseteq(b^a)\ \middle|\
|\lambda|>|\lambda^\vee|,\ \text{or }
|\lambda|=|\lambda^\vee|\text{ and }\lambda^\vee\le_{\mathrm{lex}}\lambda
\right\}.
\]
Thus $\lambda\mapsto\{\lambda,\lambda^\vee\}$ identifies $\mathcal P^{\mathrm{rep}}_{a,b}$ with the set of unordered rectangular complementary pairs, and $|\lambda|\ge ab/2$ for every $\lambda\in\mathcal P^{\mathrm{rep}}_{a,b}$.

Suppose, for contradiction, that there is a nontrivial relation
\begin{equation}\label{eq-monomial-relation}
\sum_{\lambda\in\mathcal P^{\mathrm{rep}}_{a,b}} c_\lambda m_\lambda m_{\lambda^\vee}=0
\end{equation}
over $\FF$.

Assume first that some nonzero coefficient occurs with $|\lambda|>ab/2$. Set
\[
e=\max\{|\lambda|\mid
\lambda\in\mathcal P^{\mathrm{rep}}_{a,b},\ c_\lambda\ne0\}.
\]
If $|\lambda|=e$, then $|\lambda^\vee|=ab-e<e$.  By Lemma~\ref{lem-monomial-top-p}, write $m_\lambda=\kappa_\lambda p_e+Q_\lambda$, where $\kappa_\lambda\ne0$ and $Q_\lambda\in\FF[p_1,\ldots,p_{e-1}]$.  Since $m_{\lambda^\vee}$ is homogeneous of degree $ab-e<e$, its power-sum expansion does not involve $p_e$.  Hence $[p_e]_{\mathrm p}(m_\lambda m_{\lambda^\vee}) =\kappa_\lambda m_{\lambda^\vee}$.

If $\mu\in\mathcal P^{\mathrm{rep}}_{a,b}$ and $|\mu|<e$, then $|\mu^\vee|\le|\mu|<e$, so neither $m_\mu$ nor $m_{\mu^\vee}$ involves $p_e$ in its power-sum expansion.  Applying $[p_e]_{\mathrm p}$ to \eqref{eq-monomial-relation} gives
\[
\sum_{\substack{\lambda\in\mathcal P^{\mathrm{rep}}_{a,b}\\ |\lambda|=e}}
c_\lambda\kappa_\lambda m_{\lambda^\vee}=0.
\]
The complements $\lambda^\vee$ appearing in this sum are distinct, and the monomial symmetric functions are linearly independent in $\Lambda_\FF$. Thus $c_\lambda=0$ for every $\lambda\in\mathcal P^{\mathrm{rep}}_{a,b}$ with $|\lambda|=e$, contradicting the definition of $e$.

Therefore no nonzero term in \eqref{eq-monomial-relation} has $|\lambda|>ab/2$.  Since $|\lambda|\ge ab/2$ for $\lambda\in\mathcal P^{\mathrm{rep}}_{a,b}$, if $ab$ is odd then every coefficient in \eqref{eq-monomial-relation} is zero, contrary to the assumption that it is nontrivial.  Hence assume $ab$ is even, and put $e=ab/2$.  Every remaining nonzero term in \eqref{eq-monomial-relation} has $|\lambda|=|\lambda^\vee|=e$.

For such $\lambda$, Lemma~\ref{lem-monomial-top-p} gives $m_\lambda=\kappa_\lambda p_e+Q_\lambda$ and $m_{\lambda^\vee}=\kappa_{\lambda^\vee}p_e+Q_{\lambda^\vee}$, where $\kappa_\lambda,\kappa_{\lambda^\vee}\ne0$ and neither $Q$-term involves $p_e$.  Applying $[p_e]_{\mathrm p}$ to \eqref{eq-monomial-relation} gives
\begin{equation}\label{eq-monomial-middle-coeff}
\sum_{\substack{\lambda\in\mathcal P^{\mathrm{rep}}_{a,b}\\ |\lambda|=e}}
c_\lambda\bigl(\kappa_\lambda Q_{\lambda^\vee}
+\kappa_{\lambda^\vee}Q_\lambda\bigr)=0.
\end{equation}

Let $(\Lambda_\FF)_e$ be the homogeneous component of degree $e$.  The monomial symmetric functions $m_\alpha$, $\alpha\vdash e$, form a basis of $(\Lambda_\FF)_e$.  Since $p_e=m_{(e)}$, the subspace $\FF p_e$ is the span of one basis vector.  Hence the cosets $m_\alpha+\FF p_e$, with $\alpha\vdash e$ and $\alpha\ne(e)$, form a basis of $(\Lambda_\FF)_e/\FF p_e$.

For each $\alpha\vdash e$, Lemma~\ref{lem-monomial-top-p} gives $Q_\alpha=m_\alpha-\kappa_\alpha p_e$, so $Q_\alpha+\FF p_e=m_\alpha+\FF p_e$.  Passing \eqref{eq-monomial-middle-coeff} to $(\Lambda_\FF)_e/\FF p_e$ gives
\begin{equation}\label{eq-monomial-middle-quotient}
\sum_{\substack{\lambda\in\mathcal P^{\mathrm{rep}}_{a,b}\\ |\lambda|=e}}
c_\lambda\bigl(
\kappa_\lambda(m_{\lambda^\vee}+\FF p_e)
+\kappa_{\lambda^\vee}(m_\lambda+\FF p_e)
\bigr)=0.
\end{equation}
We now show that every coefficient $c_\lambda$ appearing in \eqref{eq-monomial-middle-quotient} is zero, except possibly the coefficient of the pair $\{(e),(e)\}$. Fix $\lambda\in\mathcal P^{\mathrm{rep}}_{a,b}$ with $|\lambda|=e$ and $c_\lambda\ne0$.

If $\lambda\ne\lambda^\vee$ and neither partition is $(e)$, then the coset $m_\lambda+\FF p_e$ appears in the summand indexed by $\lambda$ with coefficient $c_\lambda\kappa_{\lambda^\vee}$.  It appears in no other summand, because the only complementary pair containing $\lambda$ is $\{\lambda,\lambda^\vee\}$, and $\mathcal P^{\mathrm{rep}}_{a,b}$ contains only one representative of this pair.  Hence $c_\lambda\kappa_{\lambda^\vee}=0$, so $c_\lambda=0$.

If the pair indexed by $\lambda$ is $\{(e),\mu\}$ with $\mu\ne(e)$, then $m_{(e)}+\FF p_e=0$, while $m_\mu+\FF p_e$ is a basis vector.  This basis vector appears only in the summand indexed by $\lambda$, with coefficient $c_\lambda\kappa_{(e)}$.  Hence $c_\lambda=0$.

If $\lambda=\lambda^\vee\ne(e)$, then the coset $m_\lambda+\FF p_e$ appears only in the summand indexed by $\lambda$, with coefficient $2c_\lambda\kappa_\lambda$.  Since $\operatorname{char}\FF=0$ and $\kappa_\lambda\ne0$, this gives $c_\lambda=0$.

The preceding coefficient comparisons show that $c_\lambda=0$ for every $\lambda\in\mathcal P^{\mathrm{rep}}_{a,b}$ with $|\lambda|=e$, except possibly when $\lambda=(e)$ and $\lambda^\vee=(e)$.  Returning to \eqref{eq-monomial-relation} after discussing the above cases, either no term remains, or the relation reduces to $c_{(e)}m_{(e)}^2=0$.  In the latter case, $c_{(e)}=0$ because $m_{(e)}\ne0$ and $\Lambda_\FF$ is an integral domain.  Hence every coefficient in \eqref{eq-monomial-relation} is zero, which is a contradiction.  Therefore the products are linearly independent over $\FF$.

It remains to prove the integral statement.  Taking $\FF=\QQ$, any $\ZZ$-linear relation among the same products in $\Lambda_\ZZ$ becomes a $\QQ$-linear relation in $\Lambda_\QQ$ after extension of scalars.  Hence all of its coefficients are zero.
\end{proof}

\begin{remark}
The characteristic-zero hypothesis is essential in Theorem~\ref{thm-monomial}. It is used to treat the power sums as polynomial coordinates, so that coefficient extraction with respect to a single $p_e$ is available.  It also enters Lemma~\ref{lem-monomial-top-p}, where one divides by $r_\alpha$, and in the equal-degree case of the proof of Theorem~\ref{thm-monomial}, where the term $2c_\lambda\kappa_\lambda m_\lambda$ occurs.  In positive characteristic, complementary monomial products can be linearly dependent. For example, in characteristic $2$,
\[
m_{(1)}^2=m_{(2)}+2m_{(1,1)}=m_{(2)}.
\]
Thus, for the one-row rectangle $(2)$, the complementary products $m_\varnothing m_{(2)}$ and $m_{(1)}m_{(1)}$ are equal.
\end{remark}

\printbibliography

\end{document}